\documentclass{amsart}
\usepackage{amsmath,amsfonts,amssymb,amsthm,xcolor}
\usepackage{latexsym, xspace, enumerate}
\usepackage[mathscr]{eucal}
\usepackage[all]{xypic}
\usepackage{hyperref}

\usepackage{vmargin}

\setmarginsrb{28mm}{20mm}{28mm}{25mm}%
          {10mm}{7mm}{10mm}{10mm}

\theoremstyle{theorem}
\newtheorem{theorem}{Theorem}[section]

\newtheorem{question}[theorem]{Question}
\newtheorem{remark}[theorem]{Remark}
\newtheorem{proposition}[theorem]{Proposition}
\newtheorem{lemma}[theorem]{Lemma}

\newtheorem{corollary}[theorem]{Corollary}

\theoremstyle{definition}
\newtheorem{definition}[theorem]{Definition}

\numberwithin{equation}{section}

\newcommand{\R}{\mathbb R}
\newcommand{\pa}{\overset{\leftarrow}{\mathrm{par}}}
\newcommand{\bik}{\mathrm{bik}}
\newcommand{\N}{\mathbb N}
\renewcommand{\c}{\mathcal C}

\newcommand{\Z}{\mathbb Z}

\def\B{\mathcal B}

\def\End{\mathrm{End}}

\def\Im{\mathrm{Im}}
\def\nub{\mathrm{nub}}

\def\normal{\mathrel{\unlhd}}

\global\def\f{\phi}

\def\U{\mathcal U}

\newcommand{\tdlc}{t.d.l.c.\ }

\numberwithin{equation}{section}

\title[Topological entropy in t.d.l.c. groups]{Topological entropy in totally disconnected\\ locally compact groups}

\author{Anna Giordano Bruno}
\address{Anna Giordano Bruno - Dipartimento di Matematica e Informatica,
Universit\`{a} di Udine\\Via delle Scienze 206, 33100 Udine, Italy}
\curraddr{}
\email{anna.giordanobruno@uniud.it}
\thanks{The first named author is supported by Programma SIR 2014 by MIUR, Project GADYGR, Number RBSI14V2LI, cup G22I15000160008}

\author{Simone Virili}
\address{Simone Virili - Dipartimento di Matematica Pura e Applicata, 
Universit\`{a} di Padova\\Via Trieste 63, 35121 Padova, Italy}
\curraddr{}
\email{virili.simone@gmail.com}
\thanks{The second named author was partially
 supported by Fondazione Cassa di Risparmio di Padova e Rovigo (Progetto di Eccellenza ``Algebraic structures and their applications''), and by the projects DGI MINECO MTM2011-28992-C02-01 and MINECO MTM2014-53644-P (Spain)}

\begin{document}


\begin{abstract}
Let $G$ be a topological group, let $\phi$ be a continuous endomorphism of $G$ and let $H$ be  a closed $\phi$-invariant subgroup of $G$. We study whether the topological entropy is an additive invariant, that is, $$h_{top}(\phi)=h_{top}(\phi\restriction_H)+h_{top}(\bar\phi)\,,$$ where $\bar\phi:G/H\to G/H$ is the map induced by $\phi$. We concentrate on the case when $G$ is locally compact totally disconnected and $H$ is either compact or normal. Under these hypotheses, we show that the above additivity property holds true whenever $\phi H=H$ and $\ker(\phi)\leq H$.
As an application we give a dynamical interpretation of the scale $s(\phi)$, by showing that $\log s(\phi)$ is the topological entropy of a suitable map induced by $\phi$. Finally, we give necessary and sufficient conditions for the equality $\log s(\phi)=h_{top}(\phi)$ to hold.
\end{abstract}

\subjclass[2010]{37B40, 22D05, 22D40, 54H11, 54H20, 54C70}
\keywords{Topological entropy, totally disconnected locally compact group, scale, continuous endomorphism, Addition Theorem}

\maketitle

\section{Introduction}

Topological entropy for continuous self-maps of compact spaces was introduced in \cite{AKM} by Adler, Konheim and McAndrew, in analogy with the measure entropy studied in ergodic theory by Kolmogorov and Sinai.
In his celebrated paper \cite{B}, Bowen gave a definition of entropy for uniformly continuous self-maps of metric spaces. Later on, Hood in \cite{hood} extended Bowen's entropy to uniformly continuous self-maps of uniform spaces. This notion of entropy is sometimes called \emph{uniform entropy}, and it coincides with the topological entropy in the compact case (when the given compact topological space is endowed with the unique uniformity compatible with the topology). For this reason we call \emph{topological entropy} also Hood's extension and we denote it by $h_{top}$ (see \S\ref{htop-sec} for a definition). 

\medskip
Let $G$ be a topological group and let $\phi:G\to G$ be a continuous endomorphism. When endowed with its left uniformity $\U$, then $(G,\U)$ is a uniform space, and $\phi$ is uniformly continuous with respect to $\U$. Hence, Hood's definition of the topological entropy $h_{top}$ applies to any given continuous endomorphism $\phi:G\to G$. Similarly, if $H$ is a closed subgroup of $G$, the set $G/H$ of the left cosets of $H$ in $G$ inherits from $G$ a natural uniform structure $\bar\U$ (see \S\ref{tdlcq}), that we call left uniformity of $G/H$ and that generates the quotient topology of $G/H$. If $H$ is $\phi$-invariant, the map $\bar\phi:G/H\to G/H$ induced by $\phi$ is uniformly continuous with respect to $\bar\U$, so $h_{top}(\bar\phi)$ is defined. 

\medskip
In this paper we study the following general question (see \cite[Question 4.3]{DSV} in the locally compact Abelian case) when $G$ is a totally disconnected locally compact (briefly, t.d.l.c.) group.

\begin{question}\label{ATq}
Let $G$ be a topological group, $\phi:G\to G$ a continuous endomorphism and $H$ a closed $\phi$-invariant subgroup of $G$. Is it true that 
\begin{equation}\label{AT-eq}
h_{top}(\phi)=h_{top}(\phi\restriction_H)+h_{top}(\bar\phi)\, ,
\end{equation}
where $\bar\phi:G/H\to G/H$ is the map induced by $\phi$?
\end{question}

We say that the Addition Theorem holds if the formula \eqref{AT-eq} is verified. 
Some instances of the Addition Theorem are already known. Indeed, as a consequence of \cite[Corollary 4.7]{DSV},  the Addition Theorem holds when $H$ is a normal and {\em open} subgroup of the locally compact group $G$; in fact, one can directly check that, under these strong assumptions, 
$$h_{top}(\bar\phi)=0 \ \ \ \text{ and }\ \ \ h_{top}(\phi)=h_{top}(\phi\restriction_H)\, .$$
Moreover, it is known from \cite[Theorem 4.5.8]{DG-islam} that, if $G_1$ and $G_2$ are \tdlc groups and $\phi_i:G_i\to G_i$ is a continuous endomorphism for $i=1,2$, then $h_{top}(\phi_1\times\phi_2)=h_{top}(\phi_1)+h_{top}(\phi_2)$.

\medskip
An important known case of the Addition Theorem is the compact one: when $G$ is a compact group and $H$ is a closed $\phi$-invariant normal subgroup of $G$ then \eqref{ATq} holds true. Yuzvinski proved this  in \cite{Y}  for separable compact groups (a generalization for the measure entropy was given by Thomas in \cite{Thomas}). Later on, Bowen proved in \cite[Theorem 19]{B} a version of the Addition Theorem for compact metric spaces. 
The general statement, when $G$ is compact but not necessarily metrizable, is deduced from the metrizable case in \cite[Theorem 8.3]{Dik+Manolo}.\\
Let us also remark that, after the introduction of entropy for actions of amenable groups in \cite{OrnWeiss}, there has been considerable effort to generalize Yuzvinski's Addition Theorem to this context. Some of the main steps in this development have been done, chronologically, in  \cite{LSW} (for actions of $\Z^d$), \cite{Miles} (for actions of a general countable torsion-free Abelian group, so in particular $\Z^{(\N)}$), and \cite{Li} (where Li proved a very general Addition Theorem for actions of a countable amenable group).

\medskip
As mentioned above, in this paper we consider Question \ref{ATq} for \tdlc groups. For these groups van Dantzig proved in \cite{vD} that the family
$$\B(G):=\{U\leq G: \text{$U$ compact and open}\}\,$$
is a base for the neighborhoods of $1$ in $G$.  As noticed in \cite{DSV} (see \S\ref{htop-sec} and Proposition \ref{lim}), the topological entropy of a continuous endomorphism $\phi:G\to G$ of a \tdlc group $G$ can be computed as
$$h_{top}(\phi)=\sup\{H_{top}(\phi,K):K\in\B(G)\}\,,\ \text{where}\ \ H_{top}( \phi, K)=\lim_{n\to\infty}\frac{\log [K:K_{-n}]}{n}\,;$$
here, $K_{-n}=K\cap\phi^{-1}K\cap\ldots\cap\phi^{-n}K\in\B(G)$, and the index $[K:K_{-n}]$ is finite since $K_{-n}$ is open in the compact subgroup $K$.

If $H$ is a closed $\phi$-invariant subgroup of $G$,  and $H$ is compact but not necessarily normal, we see in \S \ref{htop-sec} how Hood's definition of topological entropy applies to the map $\bar\phi:G/H\to G/H$,  obtaining the following formula (see Proposition \ref{lim}):
$$h_{top}(\bar \phi)=\sup\{H_{top}( \phi, K):H\leq K\in \B(G)\}\,.$$

The main result of this paper is the following instance of the Addition Theorem:

\begin{theorem}\label{AT}
Let $G$ be a \tdlc group, $\phi:G\to G$ a continuous endomorphism and $H$ a closed $\phi$-stable subgroup of $G$ containing $\ker(\phi)$. If $H$ is either normal or compact, then 
\begin{equation}\tag{$*$}h_{top}(\phi)= h_{top}(\phi\restriction_H)+h_{top}(\bar\phi)\,,\end{equation}
where $\bar\phi:G/H\to G/H$ is the map induced by $\phi$.
\end{theorem}

In particular, the Addition Theorem holds for topological automorphisms of \tdlc groups:

\begin{corollary}
Let $G$ be a \tdlc group, $\phi:G\to G$ a topological automorphism and $H$ a closed $\phi$-stable normal subgroup of $G$. Then $$h_{top}(\phi)= h_{top}(\phi\restriction_H)+h_{top}(\bar\phi)\,,$$
where $\bar\phi:G/H\to G/H$ is the topological automorphism induced by $\phi$.
\end{corollary}

The proof of Theorem \ref{AT} is given in \S\ref{ATproof}, where we treat separately the cases when the subgroup $H$ is normal or compact.
In fact, the proofs of these two cases, even with their technical differences, use similar ideas and follow a similar pattern, that is, we prove separately the two inequalities giving the equality in $(*)$.
While the proof of the lower bound uses relatively standard arguments, the proof of the upper bound is based on a Limit Free Formula for the computation of the topological entropy (see Proposition \ref{limit_free}). Indeed, following \cite{Willis_endo}, for every $U\in\B(G)$ we can construct a compact subgroup $U_+$ of $G$ contained in $U$
(see Definition \ref{newU+}), such that $U_+\leq \phi U_+$ and
$$H_{top}(\phi,U)=\log[\phi U_+:U_+]\,.$$
The counterpart of this formula for topological automorphisms was proved in \cite{AGB} and for compact groups in \cite{DG-limitfree}.

\medskip
In Section \ref{htopvsscale}, we show a precise relation between the topological entropy and the scale, generalizing a result from \cite{BDG}.
Indeed, in the recent paper \cite{Willis_endo}, extending the same notion from \cite{Willis}, Willis defined the \emph{scale} of a continuous endomorphism $\phi$ of a \tdlc group $G$ as the positive integer
$$s(\phi):=\min\{[\phi U: U\cap\phi U]:U\in \B(G)\}\,.$$
Moreover, a subgroup $U\in\B(G)$ is said to be \emph{minimizing} if the value $s(\phi)$ is attained at $U$, that is, $s(\phi)=[\phi U:U\cap \phi U]$, and $\nub(\phi):=\bigcap\{U\in\B(G): U\ \text{is minimizing}\}$ is a compact $\phi$-stable subgroup of $G$.
We see in Proposition \ref{scale} that 
\begin{equation}\label{s=h}\log s(\phi)=h_{top}(\bar\phi)\,,\end{equation}
where $\bar\phi:G/\nub(\phi)\to G/\nub(\phi)$ is the map induced by $\phi$. Moreover, we describe $\nub(\phi)$ in dynamical terms that depend only on $G$ and $\phi$, and not on the scale.
A consequence of Theorem \ref{AT} and of \eqref{s=h} is that $\log s(\phi)=h_{top}(\phi)$ if and only if $h_{top}(\phi\restriction_{\nub(\phi)})=0$, if and only if $\nub(\phi)=\{1\}$.

\subsection*{Conventions and notation}

All topological groups in this paper are Hausdorff.

\noindent
We denote by $\N$ and $\N_{>0}$ respectively the set of natural numbers and the set of positive integers. Analogously, $\R$ and $\R_{>0}$ stand respectively for the real numbers and the positive real numbers.


\noindent
For a group $G$ and an endomorphism $\phi:G\to G$, we say that a subgroup $H$ of $G$ is {$\phi$-stable} if $\phi H=H$ and {$\phi$-invariant} if $\phi H \leq H$.

\noindent
For a group $G$ and a subgroup $H$ of $G$, we denote by $G/H=\{xH:x\in G\}$ the set of all left cosets of $H$ in $G$, and by $[G:H]$ the index of $H$ in $G$, that is the size of $G/H$. If $K$ is another subgroup of $G$, then $KH/H$ is the family of all left cosets of $H$ in $G$ with representing elements in $K$, that is $KH/H=\{kH:k\in K\}$, and we denote by $[KH:H]$ the size of this family, generalizing the usual notion of index. 

\noindent
Moreover, $N_G(H)=\{x\in G: x^{-1}Hx=H\}$ is the normalizer of $H$ in $G$. We say that a subgroup $L$ of $G$ normalizes $H$ precisely when $L\leq N_G(H)$; equivalently, $x^{-1}Hx\subseteq H$ for every $x\in L$.

\noindent
For a topological group $G$, we denote by $\End(G)$ the semigroup of all the continuous endomorphisms of $G$.

\section{Background and preliminary results}\label{bg}

\subsection{Locally compact groups and their quotients}\label{tdlcq}

A topological group $G$ can be always endowed with a natural uniform structure $\U$, called the \emph{left uniformity} of $G$ (for every $g\in G$ the multiplication $x\mapsto gx$ is uniformly continuous with respect to $\U$), which generates the given topology of $G$. If $\B$ is a base for the neighborhoods of $1$ in $G$, the family 
$$\mathcal V:=\{U_K:K\in\B\}\, ,\text{ where }\ U_K:=\{(x, y) : y^{-1}x \in K\}\,,$$ is a fundamental system of entourages of $\U$. 

Similarly, if $H$ is a closed subgroup of $G$, then $G/H$ inherits from $G$ a natural uniform structure $\bar\U$, for which a fundamental system of entourages is given by  the family 
$$\bar{\mathcal V}=\{\bar U_K:K\in\B\}\,,\text{ where }\ \bar U_K:=\{(xH,yH):  y^{-1}x \in K \}\,. $$
The topology generated by $\bar\U$ on $G/H$ coincides with the quotient topology of $G/H$. Furthermore, $G$ acts on $G/H$ on the left, in the sense that, to each element $g\in G$, we can associate the following uniform automorphism of $G/H$:
$$\lambda_g:G/H\to G/H\,\ \ \text{such that}\ \ xH\mapsto gxH\,.$$
In fact, $\lambda_g\lambda_h=\lambda_{gh}$, for all $g,\, h\in G$ and the inverse of $\lambda_g$ is $\lambda_{g^{-1}}$.
For this reason we call $\bar\U$ the left uniformity of $G/H$.

\medskip
A topological group $G$ is locally compact precisely when the family
$$\c(G):=\{K:K\text{ compact neighborhood of $1$ in $G$}\}$$
is a base for the neighborhoods of $1$ in $G$. In this case, one can take $\B=\c(G)$ in the definition of the fundamental systems of entourages for the uniformities $\U$ and $\bar\U$ above. Moreover, if $G$ is locally compact and $H$ is a closed subgroup of $G$, then $G/H$ is a locally compact Hausdorff topological space (see, for example, the discussion in \cite[Section 3.1]{Reiter} or \cite[Theorems 5.21, 5.22]{HR}). 
If in addition $G$ is totally disconnected, the quotient $G/H$ is $0$-dimensional (see \cite[Theorem 7.11]{HR}).

\medskip
Given a locally compact group $G$, it is known that there exists a left Haar measure $\mu$ on $G$. 
For a compact subgroup $C$ of $G$ and a relatively open subgroup $K$ of $C$, we can write $C=\dot\bigcup_{cK\in C/K}cK$. By the compactness of $C$, and since each $cK$ is open in $C$, the index $[C:K]$ is finite; so, since $\mu$ is left invariant,
\begin{equation}\label{meas_quot_eq1}
\mu(C)=[C:K]\mu(K)\,.
\end{equation}

Choose now a closed subgroup $H$ of $G$. In analogy to the left invariance of $\mu$, a measure $\bar\mu$ on $G/H$ is said to be {\em left invariant} if, for any measurable subset $C$ of $G/H$, 
$$\bar \mu(\lambda_gC)=\bar \mu(C)\ \ \text{for all $g\in G$\,.}$$
We would like to find a left invariant measure $\bar \mu$ on $G/H$, which is finite on the compact subsets of $G/H$ and  such that there exists a compact subset $K_0$ of $G/H$ with $\bar\mu(K_0)>0$. Unfortunately, such a measure does not always exist. In fact, a necessary and sufficient condition for its existence is that the restriction to $H$ of the modular function $\Delta_G$ of $G$ coincides with the modular function $\Delta_H$ of $H$ (see \cite[Corollary 3 on p.140]{Leopoldo} or \cite[Section 8.1]{Reiter}). On the other hand, if $H$ is compact, then both $\Delta_H(H)$ and $\Delta_G(H)$ are compact (multiplicative) subgroups of $\R_{>0}$, hence, $\Delta_H(H)=\Delta_G(H)=\{1\}$. We obtain the following

\begin{lemma}\label{Haar_sul_quot}
Let $G$ be a locally compact group and $H$ a compact subgroup of $G$. Then there exists a left invariant measure $\bar \mu$ on $G/H$, which is finite on the compact subsets of $G/H$ and  such that there exists a compact subset $K_0$ of $G/H$ with $\bar\mu(K_0)>0$.
\end{lemma}

In the hypotheses of the above lemma, let $\pi:G\to G/H$ be the canonical projection.
Analogously to the discussion that leads to \eqref{meas_quot_eq1}, if $C$ is a compact subgroup of $G$ and $K$ is a relatively open subgroup of $C$ containing $H$, then $\pi C=\dot\bigcup_{xK\in C/K}\pi(xK)=\dot\bigcup_{xK\in C/K}\lambda_x(\pi K)$ in $G/H$. Thus, we still have the formula
\begin{equation}\label{meas_quot_eq}\bar \mu(\pi C)=[C:K]\bar \mu(\pi K)\,.\end{equation}

\subsection{T.d.l.c. groups}

If $G$ is a \tdlc group, then, as observed in the Introduction, the subfamily 
$$\B(G)=\{U\leq G:\text{$U$ compact and open}\}$$ 
of $\c(G)$ is a base for the neighborhoods of $1$ in $G$. 
\\ For $C$ a closed subgroup of $G$, let 
$$\B(G,C):=\{U\in\B(G):C\leq U\}\,.$$

The following results will be useful many times in what follows.

\begin{lemma}\label{magic}
Let $G$ be a \tdlc group, $C$ a compact subgroup of $G$ and $K\in \B(G)$. Then there exists $L\in\B(G)$ such that $L\leq K$ and $C\leq N_G(L)$. In particular, $CL\in\B(G)$.
\end{lemma}
\begin{proof}
Let $L=\bigcap \{x^{-1}Kx:x\in C\}$. It is clear that $L\leq K$ and that $L$ (being defined as an intersection of closed subgroups) is a closed subgroup of $K$, so it is compact. Let us show that $C$ normalizes $L$, that is, $y^{-1}Ly\leq L$ for all $y\in C$. Let $y\in C$ and $n\in L$. For $x\in C$, we have $x y^{-1}nyx^{-1}=(yx^{-1})^{-1}n(yx^{-1})\in K$. 
Therefore, $y^{-1}ny\in x^{-1}Kx$ for every $x\in C$, and hence $y^{-1}ny\in L$.
\\
It remains to show that $L$ is open. Indeed, given $x\in C$, choose $U\in \B(G)$ such that $U\leq K$ and $x^{-1}Ux\leq K$. Let $W_x=Ux$ and $V_x=U$. Thus,  $W^{-1}_xV_xW_x=x^{-1}UUUx=x^{-1}Ux\leq K$. The family $\{W_x:x\in C\}$ is an open cover $C$, which is compact, so there is a finite subset $F$ of $C$ such that $C\subseteq \bigcup\{W_x:x\in F\}$. Set $V=\bigcap\{V_x:x\in F\}$. Then $x^{-1}Vx\leq K$ for each element $x\in C$; in fact, given $x\in C$, there exists $f\in F$ such that $x\in W_f$, so that $x^{-1}Vx\subseteq W_f^{-1}V_fW_f\subseteq K$. Thus, $V\leq L$, showing that $L$ is open.
\end{proof}

\begin{corollary}\label{lastcor}
Let $G$ be a \tdlc group and $C$ a compact subgroup of $G$. Then:
\begin{enumerate}[\rm (1)]
\item $\B=\{U\in \mathcal B(G): C\leq N_G(U)\}$ is a base for the neighborhoods of $1$ in $G$.
\item $\B'=\{CU: U\in\B(G),\ C\leq N_G(U)\}$ (and so also $\B(G,C)$) is a base for the neighborhoods of $C$ in $G$.
\end{enumerate}
\end{corollary}
\begin{proof}
(1) follows directly from Lemma \ref{magic}.

\smallskip\noindent
(2) Let $A$ be an open subset of $G$ containing $C$. In particular, $A$ is an open neighborhood of every element of $C$, and so for every $c\in C$, there exists $V_c\in \mathcal B(G)$ such that $cV_c\subseteq A$. Since $C\subseteq \bigcup_{c\in C}cV_c$, by the compactness of $C$ there exists a finite subset $F$ of $C$ such that $C\subseteq \bigcup_{f\in F}fV_f$. Let $K=\bigcap_{f\in F}V_f\in \B(G)$; we claim that $CK\subseteq A$. Indeed, for all $c\in C$ there exists $f\in F$ such that $c\in fV_f$, hence $cK\subseteq fV_fK=fV_f\subseteq A$. By Lemma \ref{magic} there exists $L\in \B(G)$ such that $L\leq K$ and $CL\in\B(G)$. Clearly, $CL\subseteq A$ and $CL\in\B'\subseteq\B(G,C)$.
\end{proof}

The next lemma generalizes part (2) of Corollary \ref{lastcor}.

\begin{lemma}\label{Mbase}
Let $G$ be a \tdlc group and $C$ a compact subgroup of $G$. If $\B\subseteq \B(G,C)$ is a downward directed family with respect to inclusion such that $\bigcap \B=C$, then $\B$ is a base for the neighborhoods of $C$ in $G$.
\end{lemma}
\begin{proof}
We should verify that, given a neighborhood $N$ of $C$ in $G$, there is $W\in\B$ such that $W\leq N$. Since $\B(G,C)$ is a base for the neighborhoods of $C$ in $G$ by Corollary \ref{lastcor}(2), we can suppose that $N\in\B(G,C)$. Let $U\in \B$, let $\B_U=\{W\in \B:W\leq U\}$, and consider the open cover $N\subseteq \bigcup_{x\in N}xU$. Since $N$ is compact, there is a finite subset $1\in F\subseteq N$ such that $N\subseteq \bigcup_{x\in F}xU=V$, where $V$ is compact. Then, $V\setminus N$ is compact and 
$$V\setminus N\subseteq V\setminus C=V\setminus \bigcap_{W\in \B_U}W= \bigcup_{W\in \B_U}(V\setminus W)\,.$$
By compactness of $V\setminus N$, there exist $W_1,\dots,W_n\in \B_U$ such that $V\setminus N\subseteq \bigcup_{i=1}^n(V\setminus W_i)$. Since $\B$ is downward directed, let $W$ be any element of $\B$ which is contained in $\bigcap_{i=1}^nW_i$ and notice that $W\subseteq N$. 
\end{proof}

The next corollary describes $\B(H)$ for any closed subgroup $H$ of a \tdlc group $G$, and it gives suitable subbases of $\B(G/H)$ respectively when $H$ is normal and when $H$ is compact.

\begin{corollary}\label{base_sub}
Let $G$ be a \tdlc group and $H$ a closed subgroup of $G$. Then:
\begin{enumerate}[\rm (1)]
\item $\B(H)=\{U\cap H:U\in \B(G)\}$;
\item if $H$ is normal and $\pi:G\to G/H$ is the canonical projection,  then $\pi\B(G)=\{\pi U:U\in \B(G)\}\subseteq \B(G/H)$ is a base for the neighborhoods of $H$ in $G/H$;
\item $H$ is compact if and only if $\B(G,H)$ is a base for the neighborhoods of $H$ in $G$.
\end{enumerate}
\end{corollary}
\begin{proof}
(1) The inclusion $\{U\cap H:U\in \B(G)\}\subseteq \B(H)$ is clear. On the other hand, given $V\in\B(H)$, there exists an open subset $U'\subseteq G$ such that $U'\cap H=V$. Now, $V$ is a compact subgroup of $G$ so, by Corollary \ref{lastcor}(2), there exists $U\in \B(G)$ such that $V\subseteq U\subseteq U'$. Clearly, $U\cap H=V$.

\smallskip\noindent
(2)  Let $V\in\B(G/H)$. Then $\pi^{-1}V$ is an open subgroup of $G$, thus there is $U\in \B(G)$ such that $U\leq \pi^{-1}V$. It is now clear that, $\pi U\leq V$ and $\pi U\in \B(G/H)$.

\smallskip\noindent
(3) If $H$ is compact, then $\B(G,H)$ is a base for the neighborhoods of $H$ in $G$ by Corollary \ref{lastcor}(2). Conversely, if $\B(G,H)$ is a base for the neighborhoods of $H$ in $G$, in particular it is not empty, and we can take some $U\in\B(G,H)$; then $H$ is closed in the compact $U$ and so $H$ is compact.
\end{proof}

\subsection{Indices of subgroups}

As underlined in the previous subsection, the study of Haar measure of subgroups reduces to some extent to the study of indices of subgroups. In this subsection we collect some facts about indices of subgroups of an abstract group.

\begin{lemma}\label{gi}
Let $G$ be a group and $H$, $K$, $L$ subgroups of $G$ with $H\leq K$. Then:
\begin{enumerate}[\rm (1)]
\item $[G:H]=[G:K][K:H]$;
\item $[LH:H]=[L:H\cap L]$;
\item $[K:H]\geq[K\cap L:H\cap L]$;
\item $[K:H]\geq[KL:HL]$, provided $HL=LH$;
\end{enumerate}
Let $G'$ be a group, $H'$, $K'$ subgroups of $G'$ with $H'\leq K'$, and consider a homomorphism $\phi:G'\to G$. Then:
\begin{enumerate}[\ \ \, ]
\item[\rm (5)] $[\phi^{-1} K : \phi^{-1}H] = [K\cap \Im(\phi): H\cap \Im(\phi)]$, in particular $[\phi^{-1} K : \phi^{-1}H] \leq  [K: H]$;
\item[\rm (6)] $[K'\ker(\phi): H'\ker(\phi)]= [\phi K' : \phi H']$, in particular $[K': H']\geq [\phi K' : \phi H']$.
\end{enumerate}
\end{lemma}

As a consequence of the above lemma we obtain: 

\begin{corollary}\label{im_gi}
Let $G$ be a group, $\phi\in\End(G)$ and $H\leq K$ subgroups of $G$ such that $H\leq \phi H$ and $K\leq \phi K$. If $[\phi K:\phi H]<\infty$, then $[\phi H:H]\geq [\phi K:K]$.
\end{corollary}
\begin{proof}
By Lemma \ref{gi}(1,6), $[\phi K:H]=[\phi K:K]\cdot[K:H]\geq [\phi K:K][\phi K:\phi H]$. Similarly, $[\phi K:H]=[\phi K:\phi H][\phi H:H]$. Thus,
$[\phi K:\phi H][\phi H:H]\geq [\phi K:K][\phi K:\phi H]$, and hence $[\phi H:H]\geq [\phi K:K]$.
\end{proof}

Consider now a group $B$, and let $B'\leq B$ and  $A\normal B$; we obtain the following diagram:
\begin{equation*}
\xymatrix{
& 1\ar[d] & 1 \ar[d] & 1 \ar[d] & \\
1\ar[r] & A\cap B'\ar[d]\ar[r] & B' \ar[d]\ar[r] & B'A/A\ar[r]\ar[d] & 1 \\
1\ar[r] & A\ar[d]\ar[r] & B \ar[d]\ar[r] & B/A\ar[r]\ar[d] & 1 \\
 & A/A\cap B' & B/B'  & B/B'A &  \\
}
\end{equation*}
If the groups involved in the above diagram are Abelian, then an easy application of the Snake Lemma gives that $|B/B'|=|A/A\cap B'|\cdot | B/B'A|$, as these groups fit into a suitable (short) exact sequence. In the following lemma we generalize this fact to the non-Abelian situation. In fact, we need to work in a slightly more general setting than that in the above picture, that is, we want to just assume that $B'A$ is a subgroup of $B$ (i.e., $B'A=AB'$), but allowing $A$ not to be normal in $B$. 

\begin{lemma}\label{snake}
Let $B$ be a group and $A$, $B'$ subgroups of $B$ with $A'=B'\cap A$. If $B'A=AB'$, then 
$$[B:B']= [A:A'] \cdot [B:B'A]\,.$$
\end{lemma}
\begin{proof}
It is not hard to check that the following map is well-defined and surjective:
$$\pi:B/B'\to B/B'A \ \ \ \text{such that}\ \ \ bB'\mapsto bB'A\,.$$
Hence, we should just verify that $|\pi^{-1}(bB'A)|=|A/A'|$ for all $b\in B$. This follows from the next two claims describing the fibers of $\pi$:
\begin{enumerate}[\rm ({Claim.}1)]
\item for $b_1,b_2\in B$, $\pi(b_1B')=\pi(b_2B')$ if and only if there exists $a \in A$ such that $b_1B'=b_2aB'$;
\item for $b \in B$ and $a_1,a_2\in A$, $ba_1B'=ba_2B'$ if and only if $a_1A'=a_2A'$. 
\end{enumerate} 
To verify (Claim.1), proceed as follows:
$\pi(b_1B')=\pi(b_2B')$ if and only if $b_2^{-1}b_1 \in B'A=AB'$, if and only if there exist $b\in B'$ and $a\in A$ such that $b_2^{-1}b_1=ab$, so that $b_1B'=b_2abB'=b_2aB'$, as desired.
\\
To verify (Claim.2), just notice that $ba_1B'=ba_2B'$ is equivalent to say that $a_2^{-1}a_1=a_2^{-1}b^{-1}ba_1=(ba_2)^{-1}(ba_1)\in B'$, that is, $a_2^{-1}a_1 \in A \cap B'=A'$, so that $a_1A'=a_2A'$.
\end{proof}

\subsection{Cotrajectories}

Let $X$ be a topological space and $\phi:X\to X$ a continuous self-map. Given $n\in\N$ and $U\subseteq X$, let
$$U_{-n}=U\cap \phi^{-1} U\cap\ldots\cap\phi^{-n} U\,;$$ 
the \emph{$\phi$-cotrajectory} of $U$ is
$$U_-=\bigcap_{n=0}^\infty\phi^{-n} U=\bigcap_{n=0}^\infty U_{-n}\,.$$
If $U$ is open (respectively, compact), then so is $U_{-n}$ for all $n\in\N$. Similarly, if $U$ is compact, then so is $U_-$. 

\begin{remark}
In the context of topological entropy (for example, see \cite{DG-islam}) the notations $C_n(\phi,U)$ and $C(\phi,U)$ are commonly used in place of $U_{-n}$ and $U_{-}$, that are commonly used for the study of the scale (see \cite{Willis,Willis_endo}). We adopt the shorter version, even if, in some cases, this may be slightly more ambiguous.
\end{remark}

In view of the above remark, we clarify now some notations. Let $G$ be a \tdlc group and let $\phi\in\End(G)$. Given $n\in\N$ and $U\in\mathcal B(G)$, the index $[U:U_{-n}]$ is finite (as $U$ is compact and $U_{-n}$ is open), for all $n\in\N$. 
Furthermore, given a $\phi$-invariant closed subgroup $H$ of $G$, denoting by $\bar\phi:G/H\to G/H$ the map induced by $\phi$, and letting $\pi:G\to G/H$ be the canonical projection, then for all $U\in\B(G)$
$$(U\cap H)_{-n}=(U\cap H)\cap (\phi\restriction_{H})^{-1}(U\cap H)\cap\ldots\cap(\phi\restriction_{H})^{-n}(U\cap H)$$ 
and
$$(\pi U)_{-n}=\pi U\cap \bar\phi^{-1}(\pi U)\cap\ldots\cap\bar\phi^{-n}(\pi U)\,.$$

\begin{lemma}\label{basic_cot}
Let $G$ be a \tdlc group, $\phi\in\End(G)$ and $U\in\B(G)$. For any $n\in\N$, let
$$c_n=[U:U_{-n}]\ \ \ \text{and}\ \ \ \alpha_n=[U_{-n}:U_{-n-1}]\,.$$
The following statements hold true:
\begin{enumerate}[\rm (1)]
\item $c_{n}$ divides $c_{n+1}$ for all $n\in\N$, and $\alpha_n=c_{n+1}/c_{n}$;
\item $\alpha_{n+1}\leq \alpha_n$, for all $n\in \N$;
\item the sequence $(\alpha_n)_{n\in\N}$ stabilizes.
\end{enumerate}
\end{lemma}
\begin{proof}
\noindent
(1) Since $U\geq U_{-n}\geq U_{-n-1}$, it follows from Lemma \ref{gi}(1) that 
$$[U:U_{-n-1}]=[U:U_{-n}][U_{-n}:U_{-n-1}]\,.$$
Thus, $c_{n+1}/c_n=[U_{-n}:U_{-n-1}]=\alpha_n$ is a positive integer.

\noindent
(2) For any given $n\in\N$,
$$\alpha_n=[U_{-n}:U_{-n-1}]\overset{(*)}{\geq}[\phi^{-1}U_{-n}:\phi^{-1}U_{-n-1}]\overset{(**)}{\geq} [\phi^{-1}U_{-n}\cap U:\phi^{-1}U_{-n-1}\cap U]=\alpha_{n+1}\,,$$
where $(*)$ and $(**)$ use Lemma \ref{gi}(5) and (3), respectively.

\noindent
(3) follows by (2). 
\end{proof}

In the following definition we recall some useful subgroups, namely $U_+$ and $U_n$ (for $n\in\N$), of a given $U\in\B(G)$, as introduced in \cite{Willis_endo}. The subgroups of the form $U_+$ will be crucial for the Limit Free Formula given in Proposition \ref{limit_free} (in this respect, see also Remark \ref{remark_W_traj}) and for the connection between topological entropy and scale given in Section \ref{htopvsscale}, as we briefly discussed in the Introduction.

\begin{definition}\cite{Willis_endo}\label{newU+}
Let $G$ be a \tdlc group, $\phi\in\End(G)$ and $U\in\B(G)$. Define $U_n$ inductively as follows:
\begin{enumerate}[\rm --]
\item $U_0=U$;
\item $U_{n+1}=U\cap\phi U_n$, for all $n\in\N$.
\end{enumerate}
Let also $U_{+}=\bigcap_{n\in\N}U_n$.
\end{definition}

Notice that $U_n\geq U_{n+1}\geq U_+$ and $U_n$ is compact for all $n\in\N$; similarly, $U_+$ is compact.

\begin{lemma}\label{basic_willis}\emph{\cite[Proposition 1, Lemma 2]{Willis_endo}}
Let $G$ be a \tdlc group, $\phi\in\End(G)$ and $U\in\B(G)$. The following properties hold:
\begin{enumerate}[\rm (1)]
\item $U_n=\{u\in U:\exists v\in U\text{ with $\phi^jv\in U$ for $j\in\{0,1,\ldots,n\}$ and $u=\phi^nv$}\}=\phi^n U_{-n}$ for all $n\in\N$;
\item $U_+=\{u\in U:\exists (u_n)_{n\in\N}\in U^{\N} \text{ such that $\phi(u_{n+1})=u_{n}$, for $n\geq 0$, and $u_0=u$}\}$;
\item $U_+=U\cap \phi U_+\leq \phi U_+$;
\item $\phi^kU_{-n}=U_k\cap U_{k-n}$ for all $k\leq n$ in $\N$ (in particular, $\phi^nU_{-n}=U_n$).
\end{enumerate}
\end{lemma}

Since $U_+$ is compact, so is $\phi U_+$. Furthermore, since $U$ is open, $U_+=U\cap \phi U_+$ is open in $\phi U_+$. This shows that the index $[\phi U_+:U_+]\ \text{is finite}\,.$

\begin{lemma}\label{magia_di_willis}
Let $G$ be a \tdlc group, $\phi\in\End(G)$ and $U\in\B(G)$. Then:
\begin{enumerate}[\rm (1)]
\item $U_{-n+1}\cap \phi^{-n}U_{n}=U_{-n}$;
\item $[\phi U_n:U_{n+1}]\overset{}{=}[U_{-n}:U_{-n-1}]$.
\end{enumerate}
\end{lemma}
\begin{proof}
(1) It is clear that $U_{-n}\leq U_{-n+1}$, while $U_{-n}\leq\phi^{-n}U_n$ by Lemma \ref{basic_willis}(4). On the other hand, $\phi^{-n}U_{n}\leq \phi^{-n}U$, so that
$U_{-n+1}\cap\phi^{-n}U_{n}\leq U_{-n+1}\cap\phi^{-n}U=U_{-n}$.

\smallskip\noindent
(2) Consider the map
$$\Phi:U_{-n}/U_{-n-1}\to \phi U_n/U_{n+1}\ \ \ \text{such that}\ \ \ \Phi(xU_{-n-1})=\phi^{n+1}xU_{n+1}\,.$$
Then $\Phi$ is well-defined and surjective by Lemma \ref{basic_willis}(4). Let us prove that it is injective. Indeed, choose $x,y\in U_{-n}$ such that $\phi^{n+1}xU_{n+1}=\phi^{n+1}yU_{n+1}$. This means that $\phi^{n+1}(y^{-1}x)\in U_{n+1}$, so $y^{-1}x\in U_{-n}\cap\phi^{-n-1}U_{n+1}$. By part (1), $U_{-n}\cap\phi^{-n-1}U_{n+1}=U_{-n-1}$, so that $xU_{-n-1}=yU_{-n-1}$, concluding the proof.
\end{proof}

We conclude the section with two basic lemmas that will be useful in the next section.

\begin{lemma}\label{leqU+}
Let $G$ be a \tdlc group, $\phi\in\End(G)$, $U\in\B(G)$ and $H$ a subgroup of $G$ with $H\leq U$.
\begin{enumerate}[\rm (1)]
\item If $\phi H\leq H$, then $H\leq U_{-n}$ for every $n\in\N$; in particular, $H\leq U_-$.
\item If $H\leq\phi H$, then $H\leq U_n$ for every $n\in\N$; in particular, $H\leq U_+$.
\end{enumerate}
\end{lemma}
\begin{proof}
(1) Since $\phi H\leq H$, it follows that $H\leq \phi^{-1}H$ and by induction one can verify that $H\leq\phi^{-n}H$ for every $n\in\N$. We proceed by induction to prove that $H\leq U_{-n}$ for every $n\in\N$. The case $n=0$ is the assumption. Assume that $H\leq U_{-n}$ for some $n\in\N$. Then $H\leq U_{-n}\cap \phi^{-n-1}(U)=U_{-n-1}$. This concludes the proof.

\smallskip\noindent
(2) We proceed by induction. For $n=0$ we find the hypothesis. If $H\leq U_n$ for some $n\in\N$, then $H\leq \phi H\leq \phi U_n$, so $H\leq U\cap \phi U_n=U_{n+1}$. 
\end{proof}

\begin{lemma}\label{NnormU}
Let $G$ be a \tdlc group and $\phi\in\End(G)$. Assume that $H$ is a subgroup of $G$ that normalizes a given $U\in\B(G)$.
\begin{enumerate}[\rm (1)]
\item If $H$ is $\phi$-invariant, then $H$ normalizes $U_{-n}$ for all $n\in\N$. Consequently, $H$ normalizes $U_-$.
\item If $H$ is $\phi$-stable, then $H$ normalizes $U_n$ for all $n\in\N$. Consequently, $H$ normalizes $U_+$.
\end{enumerate}
\end{lemma}
\begin{proof}
If $\{L_i\}_{i\in I}$ is a family of subgroups of $G$ such that $H$ normalizes $L_i$ for all $i\in I$, then $H$ normalizes $\bigcap_{i\in I}L_i$. Thus, it is enough to prove the first half of statements (1) and (2), as the second part follows by this observation.

\smallskip\noindent
(1) We proceed by induction on $n\in\N$. For $n=0$ there is nothing to prove. Given $n\in \N$ such that $H$ normalizes $U_{-n}$, let us show that $H$ normalizes $\phi^{-1}U_{-n}$. Indeed, given $x\in H$, since $H$ is $\phi$-invariant, $\phi(x^{-1}\phi^{-1}U_{-n}x)\subseteq\phi(x)^{-1}U_{-n}\phi(x)=U_{-n}$, and so $x^{-1}\phi^{-1}U_{-n}x\subseteq \phi^{-1}U_{-n}$. Thus, $H$ normalizes both $\phi^{-1}U_{-n}$ and $U$, so $H$ normalizes $U_{-n-1}=U\cap \phi^{-1} U_{-n}$. 

\smallskip\noindent
(2) We proceed by induction on $n\in\N$. For $n=0$ there is nothing to prove. Let $n\in \N$ and assume that $H$ normalizes $U_{n}$. We verify that $H$ normalizes $\phi U_n$. Indeed, given $x\in H$ there exists $z\in H$ such that $x=\phi(z)$, since $H$ is $\phi$-stable. Thus, $x^{-1}\phi U_n x=\phi(z^{-1}U_{n} z)=\phi U_{n}$. Thus, $H$ normalizes both $\phi U_n$ and $U$, so $H$ normalizes $U_{n+1}=U\cap \phi U_n$. 
\end{proof}

\section{Topological entropy in \tdlc groups}\label{topentintdlc}

\subsection{Entropy in uniform spaces}\label{htop-sec}

We first recall the version of Hood's definition of topological entropy that fits well for locally compact uniform spaces and then specialize it to the context of \tdlc groups.

\medskip
Let $(X,\U)$ be a locally compact uniform space and let $\phi: X \to X$ be a uniformly continuous self-map. For $V\in\U$, $x\in X$ and $n\in \N_{>0}$, let
$$D_n(\phi,V,x):= \bigcap_{k=0}^{n-1}\f^{-k}(V(\f^k x))\, .$$ 
Let $\B$ a fundamental system of entourages of $\U$, and recall that a Borel measure $m$ on $X$ is {\em $\phi$-homogeneous} if it satisfies the following conditions:
\begin{enumerate}[\rm ({Ho.}1)]
\item $m(K) < \infty$ for any compact subset $K\subseteq X$;
\item $m(K_0) > 0$ for some compact subset $K_0\subseteq X$;
\item for all $U\in\B$ there exist $V\in \B$ and $c\in\R_{>0} $ such that, for all $n\in \N_{>0}$ and all $x,y \in X$,
$$m(D_n(\phi, V,y)) \leq c \cdot m(D_n(\phi,U,x))\, .$$
\end{enumerate}

Suppose that there is a $\phi$-homogeneous measure $m$ on $X$. For all $U\in \U$ and $x\in X$, define
\begin{equation}\label{ccc}
k(\phi,U,x):=\limsup_{n\to\infty}-\frac{\log m(D_{n+1}(\phi,U,x))}{n}\, .
\end{equation}
The \emph{topological entropy} of $\phi$ can be defined by the following formula: for a given $x\in X$,
\begin{equation}\label{cccc}
h_{top}(\phi):=\sup\{k(\phi,U,x):U\in\B\}\,,
\end{equation}
It follows from (Ho.3) that the value $h_{top}(\phi)$ does not depend on the choice of $x\in X$. 

\begin{remark}
\begin{enumerate}[\rm (1)]
\item The definition of topological entropy given by Hood in \cite{hood}, following closely the ideas of Bowen in \cite{B}, applies to any uniformly continuous self-map $\phi:X\to X$ of a uniform space $(X,\U)$. For the general definition one needs to introduce the concepts of {\em separated} and {\em spanning} subsets; for this formalism we refer to \cite[Section 2]{hood} or, in the metric case, to \cite[Section 1]{B}. 
Notice also that the definition of $\phi$-homogeneous measure given by Hood slightly differs from ours, but they are easily seen to be equivalent. 
Following  the proof of \cite[Proposition 7]{B} with the obvious changes, one can show that the definition in \eqref{cccc}, when applicable, gives the same notion of entropy as the one defined by means of {separated} or {spanning} subsets.
\item The definition of topological entropy given in \eqref{cccc} applies to the case when $X$ is {\em locally compact} and there exists a $\phi$-homogeneous measure on $X$. The local compactness plays a very important role, in fact, we want $h_{top}$ to take values in $\R_{\geq0}\cup\{+\infty\}$. On the other hand, if $m(D_{n+1}(\phi,U,x))$ is infinite for all $n\in\N$, then $k(\phi,U,x)$ is $-\infty$. 
The hypothesis that $X$ is locally compact ensures that there exists $U\in\U$ such that $U(x)$ is contained in a compact, so that $m(D_n(\phi,U,x))$ is finite for all $n\in\N$, showing that $k(\phi,U,x)$ is not $-\infty$ and belongs to $\R_{\geq0}$.
\end{enumerate}
\end{remark}

\begin{lemma}\label{antim-eq}\label{fund}
In the above notation, the following properties hold true:
\begin{enumerate}[\rm (1)]
\item $k(\phi,U_2,x)\leq k(\phi,U_1,x)$, for all $x\in X$, and $U_1\subseteq U_2$ in $\B$;
\item $h_{top}(\phi)=\sup\{k(\phi,U,x):U\in\mathcal B'\}$, whenever $\B'\subseteq \B$ is a smaller fundamental system of entourages of $\U$.
\end{enumerate}
\end{lemma}

Let us now return to our setting, that is, let $G$ be a \tdlc group and $\phi\in\End(G)$. Recall from \S\ref{tdlcq} that $\mathcal V=\{U_K:K\in \B(G)\}$, where $U_K=\{(x,y):y^{-1}x\in K\}$, is a fundamental system of entourages for the left uniformity $\U$ on $G$. Furthermore, for all $K\in\B(G)$ and $x\in G$, it is straightforward to prove that, for every $K\in\B(G)$, every $x\in G$ and $n\in\N$,
\begin{equation}\label{Dn=K-n}
D_{n+1}(\phi,U_K,x)=x K_{-n}\,.
\end{equation}
The left Haar measure $\mu$ on $G$ is $\phi$-homogeneous. Indeed, it clearly satisfies (Ho.1) and (Ho.2). Moreover, $\mu$ satisfies (Ho.3) with $\B=\mathcal V$, since, by the left invariance of $\mu$ and by \eqref{Dn=K-n}, for every $K\in\B(G)$, every $x\in G$ and $n\in\N$,
$$\mu(D_{n+1}(\phi,U_K,x))=\mu(K_{-n})\,;$$
hence, in (Ho.3) for $U\in\mathcal V$ it suffices to take $V=U$ and $c=1$.

\noindent Now, for $K\in\B(G)$, by \eqref{Dn=K-n} with $x=1$, we have
\begin{equation}\label{def1}
H_{top}(\phi,K):=k(\phi,U_K,1)=\limsup_{n\to \infty}-\frac{\log \mu(K_{-n})}{n}\,,
\end{equation}
and so, as it was noticed in \cite{DSV}, 
$$h_{top}(\phi)=\sup\{H_{top}(\phi,K):K\in\B(G)\}\,.$$

We consider now the topological entropy of $\bar\phi:G/H\to G/H$, where $H$ is a compact $\phi$-invariant subgroup of $G$.
Let $\pi:G\to G/H$ be the canonical projection. Recall from \S\ref{tdlcq} that $\bar{\mathcal V}=\{\bar U_K:K\in \B(G)\}$, where $\bar U_K=\{(xH,yH):y^{-1}x\in K\}$, is a fundamental system of entourages of the left uniformity $\bar\U$ of $G/H$. In fact, a consequence Lemma \ref{base_sub}(3) is that the smaller set $\{\bar U_K: K\in \B(G,H)\}\subseteq \bar {\mathcal V}$ is a fundamental system of entourages. 
By Lemma \ref{Haar_sul_quot}, there is a left invariant measure $\bar \mu$ on $G/H$ which satisfies (Ho.1) and (Ho.2). Proceeding as in the case of $\mu$ and $G$, and noticing that for every $K\in\B(G,H)$, every $x\in G$ and $n\in\N$,
\begin{equation}\label{Dn=K-nbar}
D_{n+1}(\bar\phi,\bar U_K,xH)=\lambda_x(\pi K)_{-n}\,,
\end{equation}
the left invariance of $\bar \mu$ easily gives (Ho.3); thus, $\bar \mu$ is $\bar \phi$-homogenous. 

\noindent 
Now, for $K\in\B(G,H)$, by \eqref{Dn=K-nbar} with $x=1$, we have
\begin{equation}\label{def2}
H_{top}(\bar \phi,\pi K):=k(\bar \phi,\bar U_K,H)=\limsup_{n\to\infty}-\frac{\log \bar \mu((\pi K)_{-n})}{n}\,.
\end{equation}
Thus, the {topological entropy} of $\bar \phi$ is
$$h_{top}(\bar \phi)=\sup\{H_{top}(\bar \phi,\pi K):K\in \B(G,H)\}\,.$$

\medskip
In Proposition \ref{lim} we are going to restate the formulas \eqref{def1} and \eqref{def2} without making recourse to the measure. We need first the following lemma:

\begin{lemma}\label{piC}
Let $G$ be a \tdlc group, $\phi\in\End(G)$, $H$ a closed $\phi$-invariant subgroup of $G$ and $\pi:G\to G/H$ the canonical projection.
If $K$ is a subgroup of $G$ containing $H$ and $n\in\N$, then $\pi(K_{-n})=(\pi K)_{-n}$.
\end{lemma}
\begin{proof}
Let $xH\in G/H$. Then, $xH\in (\pi K)_{-n}=\pi K\cap \bar\phi^{-1}(\pi K)\cap\ldots\cap \bar\phi^{-n}(\pi K)$ if and only if $\phi^i(x)H=\bar\phi^i(xH)\in \pi K=K/H$ for all $i=0,\dots,n$. This means that $\phi^i(x)\in K$ for all $i=0,\dots,n$, that is, $x\in K_{-n}$; equivalently, since $H\leq K$, $\pi x=xH\in \pi(K_{-n})$.
\end{proof}

The next proposition shows in particular that the superior limits in \eqref{def1} and \eqref{def2} are limits; item (1) was already proved in \cite[Proposition 4.5.3]{DG-islam}.
Let $\log\N_{>0}=\{\log n:n\in\N_{>0}\}$.

\begin{proposition}\label{lim}
Let $G$ be a \tdlc group, $\phi\in\End(G)$ and $H$ a compact $\phi$-invariant subgroup of $G$. Then:
\begin{enumerate}[\rm (1)]
\item $H_{top}( \phi, K)=\lim_{n\to\infty}\frac{1}{n}\log [K:K_{-n}]\in\log\N_{>0}$, for all $K\in\B(G)$;
\item $H_{top}(\bar \phi,\pi K)=H_{top}( \phi, K)=\lim_{n\to\infty}\frac{1}{n}\log [K:K_{-n}]\in\log\N_{>0}$, for all $K\in\B(G,H)$.
\end{enumerate}
\end{proposition}
\begin{proof}
(1) Let $K\in\B(G)$. By \eqref{meas_quot_eq1}, $\mu(K)=[K:K_{-n}]\mu(K_{-n})$ for every $n\in\N$, and hence by \eqref{def1}
$$H_{top}( \phi, K)=\limsup_{n\to \infty}-\frac{\log\mu(K) -\log [K:K_{-n}]}{n}=\limsup_{n\to\infty}\frac{\log [K:K_{-n}]}{n}.$$
For all $n\in \N$, let $c_n=[K:K_{-n}]$ and $\alpha_n=[K_{-n}:K_{-n-1}]$. By Lemma \ref{basic_cot}(3), the sequence $(\alpha_n)_{n\in\N}$ stabilizes, that is, there exists $n\in\N$ such that $\alpha_n=\alpha_{m}=:\alpha$ for all $n\geq m$. Therefore, for every $n\geq m$, by Lemma \ref{basic_cot}(1) we have that $\alpha=c_{n+1}/c_n$, hence $c_n=\alpha^{n-m}c_m$. So, the sequence $(\log c_n/n)_{\N}$ converges to $\log\alpha$, and
by the first part of the proof we conclude that
\begin{equation}\label{logalpha}
H_{top}(\phi,K)=\lim_{n\to\infty}\frac{\log [K:K_{-n}]}{n}=\log \alpha\,.
\end{equation}

\smallskip\noindent
(2) Let $K\in \B(G,H)$. By Lemma \ref{leqU+}(1), $H\leq K_{-n}$ for every $n\in\N$. Using \eqref{meas_quot_eq} and Lemma \ref{piC}, we obtain that 
$$\bar\mu (\pi K)=[K:K_{-n}]\bar\mu((\pi K)_{-n})\,$$ for every $n\in\N$,
so by \eqref{def2}
$$H_{top}(\bar \phi,\pi K)=\limsup_{n\to\infty}-\frac{\log \bar\mu (\pi K) -\log [K:K_{-n}]}{n}=\limsup_{n\to\infty}\frac{\log [K:K_{-n}]}{n}\,;$$
in particular, $H_{top}(\bar \phi,\pi K)=H_{top}( \phi, K)$.
\end{proof}

As a consequence, we obtain the monotonicity of the topological entropy under taking quotients over compact $\phi$-invariant subgroups:
$$h_{top}(\phi)=\sup\{H_{top}(\phi,K): K\in \B(G)\}\geq \sup\{H_{top}(\phi,K): K\in \B(G,H)\}=h_{top}(\bar \phi)\,.$$
Similarly, the topological entropy is monotone under taking closed (not necessarily compact) $\phi$-invariant subgroups:

\begin{lemma}\label{traj_sub}
Let $G$ be a \tdlc group, $\phi\in\End(G)$ and $G'$ a closed $\phi$-invariant subgroup of $G$. Then:
\begin{enumerate}[\rm (1)]
\item for $U\in \B(G)$ and $n\in\N$, $U_{-n}\cap G'=(U\cap G')_{-n}$;
\item $h_{top}(\phi)\geq h_{top}(\phi\restriction_{G'})$. 
\end{enumerate}
\end{lemma}
\begin{proof}
\noindent
(1) Clearly, $(U\cap G')_{-n}\leq U_{-n}\cap G'$. On the other hand, let $x\in U_{-n}\cap G'$, that is, $\phi^{i}(x)\in U$ for all $i=0,\dots,n$ and $x\in G'$. Since $G'$ is $\phi$-invariant, $\phi^{i}(x)\in G'$ for all $i\in \N$, so that $\phi^i(x)\in U\cap G'$ for all $i=0,\dots,n$. Hence, $x\in (U\cap G')_{-n}$. 

\noindent
(2) By Corollary \ref{base_sub}(1), $\B(G')=\{U\cap G':U\in \B(G)\}$.
By item (1), Lemma \ref{gi}(2) and Proposition \ref{lim}(2),
\begin{align*}
H_{top}(\phi\restriction_{G'},U\cap G')&=\lim_{n\to\infty}\frac{[U\cap G':(U\cap G')_{-n}]}{n}= \lim_{n\to\infty}\frac{[U\cap G':U_{-n}\cap G']}{n}\\
&\leq \lim_{n\to\infty}\frac{[U:U_{-n}]}{n}=H_{top}(\phi,U)\,,\end{align*}
for all $U\in \B(G)$.
Hence, $$h_{top}(\phi)=\sup\{H_{top}(\phi,U): U\in \B(G)\}
\geq \sup\{H_{top}(\phi\restriction_{G'},U): U\in \B(G')\}=h_{top}(\phi\restriction_{G'})\,. \qedhere$$
\end{proof}

Let us state the following useful properties of the topological entropy in the case of groups. Notice that these are direct consequences of the more general Lemma \ref{antim-eq}.

\begin{lemma}\label{antim}\label{basesuff}
Let $G$ be a \tdlc group, $\phi\in\End(G)$, $H$ a compact $\phi$-invariant subgroup of $G$ and $\bar\phi:G/H\to G/H$ the map induced by $\phi$.
\begin{enumerate}[\rm (1)]
\item If $U,\, V\in\mathcal B(G)$ and $U\leq V$, then $H_{top}(\phi,V)\leq H_{top}(\phi,U)$. 
\item If $\B\subseteq \B(G)$ is a base for the neighborhoods of $1$ in $G$, then $h_{top}(\phi)=\sup\{H_{top}(\phi,U):U\in\mathcal B\}$.
\item If $\B\subseteq \B(G,H)$ is a base for the neighborhoods of $H$ in $G$, then $h_{top}(\bar \phi)=\sup\{H_{top}(\phi,U):U\in\mathcal B\}$.
\end{enumerate}
\end{lemma}

The following corollary follows from Lemma \ref{basesuff}(2,3) and Corollary \ref{lastcor}(1,2).

\begin{corollary}\label{basebis}
Let $G$ be a \tdlc group, $\phi\in\End(G)$, $H$ a compact $\phi$-invariant subgroup of $G$ and $\bar\phi:G/H\to G/H$ the map induced by $\phi$. 
Then $h_{top}(\phi)=\sup\{H_{top}(\phi,U):U\in\mathcal B(G),\ H\leq N_G(U)\}$ 
and $h_{top}(\bar\phi)=\sup\{H_{top}(\phi,U):U\in\mathcal B(G,H),\ H\leq N_G(U)\}$.
\end{corollary}

\subsection{The Limit Free Formula}\label{lf-sec}

The aim of this subsection is to prove in Proposition \ref{limit_free} a formula for the computation of the topological entropy avoiding the limit in the definition (hence, the name Limit Free Formula). 

\begin{remark}\label{remark_W_traj}
When $\phi:G\to G$ is a topological automorphism of a \tdlc group $G$, one of the main ingredients used in \cite{AGB} was the full cotrajectory $C(\phi^{-1},U)=\bigcap_{n=0}^{\infty}\phi^n U$ of the inverse $\phi^{-1}$ of $\phi$. When $\phi$ is a continuous endomorphism we need to substitute $C(\phi^{-1},U)$ and $C_n(\phi^{-1},U)$ by the smaller subgroups $U_+$ and $U_n$ (see Definition \ref{newU+}).
\end{remark}

\begin{proposition}\label{limit_free}
Let $G$ be a \tdlc group, $\phi\in\End(G)$ and $U\in \B(G)$. Then
$$H_{top}(\phi,U)=\log[\phi U_+:U_+]\,.$$
\end{proposition}
\begin{proof}
By Lemma \ref{basic_cot} there exist $n_0\in \N$ and $\alpha>0$ such that $\alpha_n=\alpha$ for any positive integer $n>n_0$, and $H_{top}(\phi,U)=\log\alpha$ by \eqref{logalpha}. Hence, it suffices to prove that
\begin{equation}\label{claim_limit_free_proof}\log[\phi U_+:U_+]=\log\alpha\,.\end{equation}
Since $U_+=U\cap \phi U_+$ by Lemma \ref{basic_willis}(3), and using Lemma \ref{gi}(2),
$$[\phi U_+:U_+]=[\phi U_+:U\cap \phi U_+]=[\phi U_+ \cdot U:U]\,.$$
Now, both $U$ and $\phi U$ are compact, so $\phi U\cdot U$ is compact as well. Thus, $[\phi U\cdot U:U]$ is finite, $U$ being open. Consequently, the sequence $\{[\phi U_n \cdot U:U]:n\in \N\}$ is a non-increasing sequence of positive integers bounded above by $[\phi U_0\cdot U:U]=[\phi U\cdot U:U]$. Therefore, this sequence stabilizes, so there exists $n_1\in\N$ such that 
\begin{equation}\label{stab_eq}\phi U_n \cdot U=\phi U_{n_1} \cdot U\ \ \text{for all $n\geq n_1$.}\end{equation}
Thus, for all $m\geq n_1$,
$$\phi U_m \cdot U=\bigcap_{n=0}^{\infty}(\phi U_{n} \cdot U)=\left(\bigcap_{n=0}^{\infty}\phi U_{n}\right) \cdot U=\phi\left(\bigcap_{n=0}^{\infty} U_{n}\right) \cdot U=\phi U_+\cdot U\,,$$
where the above equalities follow respectively by \eqref{stab_eq}, \cite[Lemma 2.3]{AGB}, and \cite[Lemma 1]{Willis_endo}. Choose now a positive integer $n\geq \max\{n_0,n_1\}$, then by Lemma \ref{gi}(2),
\begin{equation*}\begin{split}
[\phi U_+:U_+]=[\phi U_+:U\cap\phi U_+]=[\phi U_+\cdot U:U]=[\phi U_n\cdot U:U]=\\=[\phi U_n:U\cap \phi U_n]=[\phi U_{n}:U_{n+1}]=[U_{-n}:U_{-n-1}]=\alpha\,,
\end{split}\end{equation*}
where the penultimate equality comes from Lemma \ref{magia_di_willis}(2). This concludes the proof of \eqref{claim_limit_free_proof}.
\end{proof}

The following corollary is an immediate consequence of the above proposition:

\begin{corollary}\label{coro_lim_free}
Let $G$ be a \tdlc group, $\phi\in\End(G)$ and let $H$ be a compact $\phi$-invariant subgroup of $G$. Then
$$h_{top}(\bar \phi)=\sup\{\log[\phi U_+:U_+]: U\in \B(G,H)\}\,,$$
where $\bar \phi:G/H\to G/H$ is the map induced by $\phi$.
\end{corollary}

If $H$ is $\phi$-stable, the above formula can be improved as follows:

\begin{proposition}\label{phiN}
Let $G$ be a \tdlc group, $\phi\in\End(G)$ and $H$ a compact $\phi$-stable subgroup of $G$. Then
$$h_{top}(\bar \phi)=\sup\{\log[\phi M:M]: H\leq M\leq G,\ M\ \text{compact},\ M\leq \phi M,\ [\phi M:M]<\infty\}=:s\,,$$
where $\bar \phi:G/H\to G/H$ is the map induced by $\phi$.
\end{proposition}
\begin{proof}
By Corollary \ref{coro_lim_free}, $h_{top}(\bar \phi)=\sup\{\log[\phi U_+:U_+]:U\in\mathcal B(G,H)\}$. Since $H$ is $\phi$-stable, $U_+$ contains $H$ for every $U\in\B(G)$ by Lemma \ref{leqU+}(2); moreover, $U_+\leq \phi U_+$ by Lemma \ref{basic_willis}(3) and $[\phi U_+:U_+]$ is finite. Thus, $h_{top}(\bar\phi)\leq s$. 

To prove the converse inequality, let $M$ be a compact subgroup of $G$ such that $H\leq M\leq \phi M$ and $[\phi M:M]$ is finite. Since $M$ is closed in $\phi M$ and $[\phi M:M]$ is finite, $M$ is open in $\phi M$. Consequently, there exists an open subset $U$ of $G$ such that $\phi M\cap U=M$. 
By Corollary \ref{lastcor}(2), there exists $K\in\B(G)$ such that $M\leq N_G(K)$ (so $MK\in\B(G,M)$) and $M\leq MK\subseteq U$. 
By Lemma \ref{magic} there exists $N\in \B(G)$ such that $N\leq K$ and such that $\phi M$ normalizes $K$. Since also $M$ normalizes $N$, we have that $MN\in\B(G)$. Moreover, $M=\phi M\cap U\supseteq \phi M\cap MN\geq M$, hence $M=\phi M \cap MN$. By Lemma \ref{leqU+}(2), $M\leq (MN)_+$,
and so
$$[\phi(MN)_+:(MN)_+]=[\phi(MN)_+:MN\cap \phi(MN)_+]\geq [\phi M:MN\cap \phi M]=[\phi M:M]\,,$$
where the first equality holds since $(MN)_+=MN\cap\phi(MN)_+$ by Lemma \ref{basic_willis}(3), while the inequality uses part (3) of Lemma \ref{gi} as follows:
$$[\phi(MN)_+:MN\cap \phi(MN)_+]\geq [\phi(MN)_+\cap \phi M:MN\cap \phi(MN)_+\cap \phi M]= [\phi M:MN\cap \phi M]\,.$$
By the arbitrariness of $M$ we conclude that $s\leq h_{top}(\bar \phi)$.
\end{proof}

By Corollary \ref{basebis} and Proposition \ref{limit_free}, and since $H\leq N_G(U)$ implies $H\leq N_G(U_+)$ by Lemma \ref{NnormU}(2), we have $h_{top}(\bar\phi)=\sup\{\log[\phi U_+:U_+]:U\in\B(G),\ H\leq N_G(U_+)\}$. So, Proposition \ref{phiN} with $H=\{1\}$ gives
\begin{equation}\label{simplier}
h_{top}(\phi)=\sup\{\log[\phi M:M]: M\leq\phi M\leq G,\ M\ \text{compact},\ [\phi M:M]<\infty,\ H\leq N_G(M)\}\, .
\end{equation}

\subsection{Proof of the Addition Theorem}\label{ATproof}

This section is devoted to the proof of Theorem \ref{AT}, that we divide into four lemmas.  In Lemmas \ref{easy<} and \ref{difficult<} we handle the case where $H$ is a compact (not necessarily normal) subgroup of the \tdlc group $G$. Let us remark that the proof of Lemma \ref{easy<}, establishing the inequality $\geq$ in $(*)$ (see the statement of Theorem \ref{AT}), is almost self-contained. On the other hand, Lemma \ref{difficult<}, proving the converse inequality, relies on Proposition \ref{phiN}, which itself relies on the Limit Free Formula and so, indirectly, on most of the theory developed in Section \ref{bg} and the first part of Section \ref{topentintdlc}. Analogous observations can be done for Lemmas \ref{easy_normal} and \ref{difficult_normal} respectively, in which we handle the case when $H$ is a normal subgroup.

\medskip
Let us now assume that $H$ is a compact subgroup of $G$, let $\phi\in \End(G)$, and let $H$ be $\phi$-invariant. Then by Lemma \ref{base_sub}(1) and Lemma \ref{basesuff}(2,3),
$$h_{top}(\phi\restriction_H)=\sup\{H_{top}(\phi,U\cap H): U\in \B(G)\}\ \text{ and }\ h_{top}(\bar \phi)=\sup\{H_{top}(\phi,U): U\in \B(G,H)\}\,.$$

\begin{lemma}\label{easy<}
Let $G$ be a \tdlc group, $\phi\in\End(G)$ and $H$ a compact $\phi$-invariant subgroup of $G$. Then  
$$h_{top}(\phi)\geq h_{top}(\phi\restriction_H)+h_{top}(\bar\phi)\,,$$
where $\bar\phi:G/H\to G/H$ denotes the endomorphism induced by $\phi$.
\end{lemma}
\begin{proof}
Choose arbitrarily $U_1\in\B(G)$ and $U_2\in \B(G,H)$. By Corollary \ref{lastcor}(1) there exists $U\in \B(G)$ such that $U\leq U_1\cap U_2$ and $H\leq N_G(U)$ (in particular, $UH\in\B(G,H)$). Now, given $n\in\N$, since $H$ normalizes $U_{-n}$ by Lemma \ref{NnormU}(1), $(U\cap H)U_{-n}$ is a subgroup of $G$, and Lemma \ref{gi}(1) yields
$$[U:U_{-n}]=[U:(U\cap H)U_{-n}]\cdot[(U\cap H)U_{-n}:U_{-n}]\,.$$
By Lemma \ref{traj_sub}(1), $U_{-n}\cap H=(U\cap H)_{-n}$ and so, using Lemma \ref{gi}(2) for the first equality, 
$$[(U\cap H)U_{-n}:U_{-n}]=[U\cap H:U_{-n}\cap H]=[U\cap H:(U\cap H)_{-n}]\,.$$
Let $\pi:G\to G/H$ be the canonical projection. By Lemma \ref{gi}(4), 
$$[U:(U\cap H)U_{-n}]\geq [UH:H(U\cap H)U_{-n} ]=[UH:U_{-n}H]\geq [UH:(UH)_{-n}].$$
Hence, $[U:U_{-n}]\geq [U\cap H:(U\cap H)_{-n}]\cdot [UH:(UH)_{-n}]$. Taking logarithms, dividing by $n$ and passing to the limit for $n\to\infty$, by Proposition \ref{lim}(1,2) and applying Lemma \ref{basesuff}(1) for the first inequality, since $\pi(UH)=\pi U$, we obtain 
$$H_{top}(\phi\restriction_H,U_1)+H_{top}(\bar \phi,\pi U_2)\leq H_{top}(\phi\restriction_H,U)+H_{top}(\bar \phi,\pi U)\leq H_{top}(\phi,U)\leq h_{top}(\phi)\,.$$
By the arbitrariness of $U_1$ and $U_2$ we can conclude.
\end{proof}

\begin{lemma}\label{difficult<}
Let $G$ be a \tdlc group, $\phi\in\End(G)$ and $H$ a compact $\phi$-stable subgroup of $G$ such that $\ker(\phi)\leq H$. Then 
$$h_{top}(\phi)\leq h_{top}(\phi\restriction_H)+h_{top}(\bar\phi)\,,$$
where $\bar\phi:G/H\to G/H$ is the map induced by $\phi$.
\end{lemma}
\begin{proof}
Let $\pi:G\to G/H$ be the canonical projection and choose a compact subgroup $M$ of $G$ such that $M\leq \phi M$, $[\phi M:M]<\infty$, and such that $H$ normalizes $M$.  Applying Lemma \ref{snake} with $B=\phi M$, $B'=M$ and $A=\phi M\cap H$, we obtain
\begin{equation}\label{*}
[\phi M:M]=[\phi M\cap H: M\cap H]\cdot [\phi M:(\phi M\cap H)M]\,.
\end{equation}
By modularity, since $M\leq \phi M$, we get $(\phi M\cap H)M=\phi M\cap HM$; moreover, $\phi(MH)=(\phi M)H=(\phi M)HM$, so by Lemma \ref{gi}(2)
$$[\phi M:(\phi M \cap H)M]=[\phi M :\phi M \cap HM]=[(\phi M)HM:HM]=[\phi(HM):HM]\,.$$
Since $HM$ is a compact subgroup of $G$ containing $H$ such that $HM\leq \phi(HM)$, and $[\phi(HM):HM]<\infty$ by \eqref{*} and by hypothesis, it follows that $\log[\phi(MH):MH]\leq h_{top}(\bar\phi)$ by Proposition \ref{phiN}. 
\\
On the other hand, since $\ker(\phi)\leq H$ and $H=\phi H$, $\phi(M\cap H)=\phi M \cap \phi H=\phi M\cap H$. 
Thus, $[\phi M\cap H: M\cap H]=[\phi(M\cap H): M\cap H]$ is finite by \eqref{*} and by hypothesis, where $M\cap H$ is a compact subgroup of $H$ such that $M\cap H\leq \phi(M\cap H)$. By Proposition \ref{phiN}, $\log[\phi M\cap H: M\cap H]\leq h_{top}(\phi\restriction_H)$.
\\
Thus, we have proved that
\begin{equation*}\log[\phi M:M]\leq h_{top}(\phi\restriction_H)+h_{top}(\bar\phi)\end{equation*}
for any compact subgroup $M$ of $G$ such that $M\leq \phi M $, $[\phi M:M]<\infty$, and such that $H$ normalizes $M$. So we can conclude by \eqref{simplier}.
\end{proof}

In Lemmas \ref{easy_normal} and \ref{difficult_normal} we handle the case when $H$ is a closed normal subgroup of the \tdlc group $G$. Recall that in this setting, if $\phi\in \End(G)$ and $H$ is $\phi$-invariant, then by Lemma \ref{basesuff}(2) and Lemma \ref{base_sub}(1,2),
$$h_{top}(\phi\restriction_H)=\sup\{H_{top}(\phi,U\cap H): U\in \B(G)\}\ \text{ and }\ h_{top}(\bar \phi)=\sup\{H_{top}(\bar\phi,\pi U): U\in \B(G)\}\,.$$

\begin{lemma}\label{easy_normal}
Let $G$ be a \tdlc group, $\phi\in\End(G)$ and $H$ a closed $\phi$-invariant normal subgroup of $G$. Then
$$h_{top}(\phi)\geq h_{top}(\phi\restriction_H)+h_{top}(\bar\phi)\,,$$
where $\bar\phi:G/H\to G/H$ denotes the endomorphism induced by $\phi$.
\end{lemma}
\begin{proof}
Let $\pi:G\to G/H$ be the canonical projection, let $U_1,\, U_2\in \B(G)$ and $U=U_1\cap U_2$. We claim that
\begin{equation}\label{par_easy_eq}H_{top}(\phi,U)\geq H_{top}(\phi\restriction_H,U\cap H)+H_{top}(\bar \phi,\pi U) \geq H_{top}(\phi\restriction_H,U_1\cap H)+H_{top}(\bar \phi,\pi U_2)\,.\end{equation}
By the arbitrariness of $U_1$ and $U_2$, this  implies that $h_{top}(\phi)\geq h_{top}(\phi\restriction_H)+h_{top}(\bar\phi)$. Thus, we have just to check  \eqref{par_easy_eq}. In fact, the second inequality is clear by Lemma \ref{basesuff}(1), while for the first one, we proceed as follows. 
Since $H\normal G$, also $U\cap H\normal U$, so that $(U\cap H)U_{-n}$ is a subgroup of $U$ containing $U_{-n}$, for all $n\in\N$. Thus, Lemma \ref{gi}(1) yields 
$$[U:U_{-n}]=[U:(U\cap H)U_{-n}]\cdot[(U\cap H)U_{-n}:U_{-n}]\,.$$
Proceeding as in the second part of the proof of Lemma \ref{easy<}, applying Lemma \ref{gi}(6) and Lemma \ref{piC} we get
$$[U:U_{-n}]\geq [U\cap H:(U\cap H)_{-n}]\cdot [UH:(UH)_{-n}]= [U\cap H:(U\cap H)_{-n}]\cdot [\pi U:(\pi U)_{-n}]\,.$$
Taking logarithms, dividing by $n$ and passing to the limit for $n\to\infty$, by Proposition \ref{lim}(1) we obtain \eqref{par_easy_eq}.
\end{proof}

\begin{lemma}\label{difficult_normal}
Let $G$ be a \tdlc group, $\phi\in\End(G)$ and $H$ a closed $\phi$-stable normal subgroup of $G$ such that $\ker(\phi)\leq H$. Then 
$$h_{top}(\phi)\leq h_{top}(\phi\restriction_H)+h_{top}(\bar\phi)\,,$$
where $\bar\phi:G/H\to G/H$ denotes the continuous endomorphism induced by $\phi$.
\end{lemma}
\begin{proof}
The proof is analogous to the proof of Lemma \ref{difficult<}, with the further simplification that there is no need to choose an $M$ which is normalized by $H$ since, being $H$ normal, $HM$ is a subgroup of $G$. 
\end{proof}

\section{Topological entropy vs scale}\label{htopvsscale}

\subsection{Reminders on scale}

We recall that for a continuous endomorphism $\phi:G\to G$ of a \tdlc group $G$, the {\em scale} of $\phi$ is defined in \cite{Willis_endo} by $$s(\phi):=\min\{[\phi U:\phi U\cap  U]:U\in \B(G)\}\,.$$
Moreover, $U\in\B(G)$ is said to be \emph{minimizing} if $s(\phi)=[\phi U:U\cap \phi U]$. The following lemma is a consequence of some results proved in \cite{Willis_endo}:

\begin{lemma}\label{Mbasebis}
Let $G$ be a \tdlc group and let $\phi\in\End(G)$. Let also 
$$\mathcal M(G,\phi):=\{U\in\B(G):U\ \text{minimizing}\}\ \text{ and }\ \nub (\phi):=\bigcap \mathcal M(G,\phi)\,.$$
Then, $\nub(\phi)$ is a compact $\phi$-stable subgroup of $G$, and $\mathcal M(G,\phi)$ is a base for the neighborhoods of $\nub(\phi)$ in $G$.
\end{lemma}
\begin{proof}
The fact that $\nub(\phi)$ is compact and $\phi$-stable is proved in \cite[Section 9]{Willis_endo}. Furthermore, by \cite[Proposition 12]{Willis_endo}, $\mathcal M(G,\phi)$ is closed under finite intersections, in particular it is downward directed with respect to inclusion, and so the conclusion follows by Lemma \ref{Mbase}.
\end{proof}

One of the main results of \cite{Willis_endo}, extending its counterpart for topological automorphisms from \cite{Willis2}, is the following characterization of minimizing subgroups (see \eqref{m=t} below). 

\begin{definition}
Let $G$ be a \tdlc group and $\phi\in\End(G)$. A $U\in\B(G)$ is said to be:
\begin{enumerate}[\rm --]
\item {\em tidy above} if $U=U_+U_-$;
\item {\em tidy below} if $U_{++}:=\bigcup_{n\in\N}\phi^nU_+$ is closed and the sequence $\{[\phi^{n+1}U_+:\phi^{n}U_+]\}_{n\in\N}$ is constant;
\item {\em tidy} if $U$ is both tidy above and tidy below. 
\end{enumerate}
\end{definition}

Theorem 7.7 in \cite{Willis_endo} states that 
\begin{equation}\label{m=t}
\text{$U\in\B(G)$ is minimizing if and only if $U$ is tidy.}
\end{equation}

We will use the following properties of tidy subgroups, note that (2) follows from (1) and \eqref{m=t}.

\begin{lemma}\label{(B)}\emph{\cite{Willis_endo}}
Let $G$ be a \tdlc group, $\phi\in\End(G)$ and $U\in\B(G)$.
\begin{enumerate}[\rm (1)]
\item If $U$ is tidy above, then $[\phi U_+:U_+]=[\phi U:U\cap\phi U]$.
\item If $U$ is tidy, then $s(\phi)=[\phi U_+:U_+]$.
\end{enumerate}
\end{lemma}

\subsection{Reduction to surjective endomorphisms and automorphisms}

In this subsection we recall the definition of the following two subgroups from \cite{Willis_endo}, and how they can be used to reduce the computation of the scale and the topological entropy respectively to topological automorphisms and to surjective continuous endomorphisms.

\begin{definition}
Let $G$ be a \tdlc group and $\phi\in\End(G)$. Define:
\begin{enumerate}[\rm --]
\item $\pa(\phi)=\{x\in G:\text{there exists $(x_n)_{n\in\N}\subseteq G$ bounded, $x_0=x$ and $\phi(x_{n+1})=x_n$, for all $n\in\N_{>0}$}\}$;
\item $\bik(\phi)=\overline{\ker_\infty(\phi)}\cap \pa(\phi)$, where $\ker_\infty(\phi)=\bigcup_{n=1}^\infty\ker(\phi^n)$.
\end{enumerate}
\end{definition}

It is shown in \cite[Section 9]{Willis_endo} that $\bik(\phi)\leq\nub(\phi)\leq \pa(\phi)$, and in particular, 
\begin{equation}\label{kpainnub}
\ker(\phi)\cap \pa(\phi)\leq \nub(\phi)\,.
\end{equation}
Moreover, $\pa(\phi)$ is a closed $\phi$-stable subgroup of $G$ such that 
\begin{equation}\label{U++inpar}
U_{++}\leq \pa(\phi)\ \text{for all}\ U\in\B(G)\,.
\end{equation}
Similarly to $\nub(\phi)$, also $\bik(\phi)$ is a compact $\phi$-stable subgroup of $G$, but $\bik(\phi)$ is normal in $\pa(\phi)$.

\medskip
For all this section, for $G$ a \tdlc group and $\phi\in\End(G)$, let
$$\psi:=\phi\restriction_{\pa(\phi)}:\pa(\phi)\to \pa(\phi) \ \text{ and }\ \tilde\psi:\pa(\phi)/\bik(\phi)\to \pa(\phi)/\bik(\phi)\,,$$
where $\tilde\psi$ is the map induced by $\psi$. Let also $\pi:\pa(\phi)\to\pa(\phi)/\bik(\phi)$ be the canonical projection.

\begin{lemma}\label{red_scale}
Let $G$ be a \tdlc group and $\phi\in\End(G)$. Then:
\begin{enumerate}[\rm (1)]
\item $\psi$ is a surjective continuous endomorphism and $\tilde\psi$ is a topological automorphism;
\item if $U\in\B(G)$ and $V=U\cap \pa(\phi)\in\B(\pa(\phi))$, then $V_+=U_+$ and $V_-=U_-\cap \pa(\phi)$. In particular, if $U$ is tidy above for $\phi$, then $V$ is tidy above for $\psi$;
\item $s(\psi)= s(\phi)$, and $\{U\cap \pa(\phi):U\in \mathcal M(G,\phi)\}\subseteq \mathcal M(\pa(\phi),\psi)$ is cofinal with respect to $\supseteq$ (i.e., for every $V\in \mathcal M(\pa(\phi),\psi)$ there exists $U\in \mathcal M(G,\phi)$  such that $U\cap \pa(\phi)\leq V$);
\item $s(\tilde\psi)=s(\psi)$, and $\mathcal M(\pa(\phi)/\bik(\phi),\tilde\psi)=\{\pi U:U\in \mathcal M(\pa(\phi),\psi)\}$;
\item $\nub(\phi)=\nub(\psi)$ and $\pi(\nub(\psi))=\nub(\tilde\psi)$.
\end{enumerate}
\end{lemma}
\begin{proof}
(1) is proved in \cite[Section 9]{Willis_endo}.

\smallskip\noindent
(2) Clearly, $V_+\leq U_+$. Since $U_+\leq \pa(\phi)$ by \eqref{U++inpar}, it follows that $U_+\leq U\cap \pa(\phi)=V$. Since $U_+\leq\phi U_+$ by Lemma \ref{basic_willis}(3), then Lemma \ref{leqU+}(2) yields that $U_+\leq V_+$, and so $U_+=V_+$.
Furthermore, $U_-\cap \pa(\phi)=\bigcap_{n\in\N} \phi^{-n}U\cap \pa(\phi)=\bigcap_{n\in\N} \psi^{-n}V=V_-$. 
\\ For the last part of the statement, assume that $U=U_+U_-$, then by modularity
$$V=U\cap \pa(\phi)=(U_+U_-)\cap \pa(\phi)=U_+(U_-\cap \pa(\phi))=V_+V_-\,,$$
showing that $V$ is tidy above for $\psi$.

\smallskip\noindent
(3) Let $U\in \mathcal M(G,\phi)$ and let $V=U\cap \pa(\phi)$. By part (2), $V$ is tidy above for $\psi$ and $V_+=U_+$, so by Lemma \ref{(B)}(1,2),
\begin{equation}\label{reas}
s(\phi)=[\phi U_+:U_+]=[\psi V_+:V_+]\geq s(\psi)\,,
\end{equation}
showing that $s(\phi)\geq s(\psi)$. \\
%
Let now $V\in  \mathcal M(\pa(\phi),\psi)$ and in view of Corollary \ref{base_sub}(1) choose $U'\in \B(G)$ such that $V=U'\cap \pa(\phi)$. By \cite[Proposition 3.9]{Willis_endo}, there exists $n\in\N$ such that $U:=(U')_{-n}$ is tidy above for $\phi$.
Since $\pa(\phi)$ is $\phi$-stable, we have that $U\cap \pa(\phi)=V_{-n}$, where $V_{-n}$ is tidy for $\psi$ by \cite[Proposition 7.10]{Willis_endo}.
Since $(V_{-n})_+=U_+$ by item (2), and applying Lemma \ref{(B)}(1,2), it follows that
$$s(\psi)=[\psi (V_{-n})_+:(V_{-n})_+]=[\phi U_+:U_+]=[\phi U:U\cap \phi U]\geq s(\phi)\geq s(\psi)\,.$$
Thus, $s(\psi)=s(\phi)$ and $U$ is tidy for $\phi$ (note that $U\cap \pa(\phi)\leq V$). 
The inclusion $\mathcal M(\pa(\phi),\psi)\supseteq \{U\cap \pa(\phi):U\in \mathcal M(G,\phi)\}$ follows now from \eqref{reas}.
%
%

\smallskip\noindent
(4) Let $U\in\mathcal M(\pa(\phi),\psi)$. Then $\bik(\phi)=\bik(\psi)\leq\nub(\psi)\leq U$, so $\pi U\in\B(\pa(\phi)/\bik(\phi))$ and, by Lemma \ref{gi}(6), $s(\psi)=[\psi U:U\cap\psi U]=[\pi\psi U:\pi(U\cap\psi U)]=[\tilde\psi \pi U:\pi U\cap\tilde\psi \pi U]\geq s(\tilde\psi)$.
To prove the converse inequality, let $W\in\mathcal M(\pa(\phi)/\bik(\phi),\tilde\psi)$. Since $\bik(\phi)\normal\pa(\phi)$ and $\bik(\phi)$ is compact, $\pi^{-1}W\in\B(\pa(\phi))$
. Moreover, by Lemma \ref{gi}(5) and since $\bik(\phi)$ is $\phi$-stable, $$s(\tilde\psi)=[\tilde\psi W:W\cap \tilde\psi W]=[\pi^{-1}(\tilde\psi W):\pi^{-1}(W\cap\tilde\psi W)]=[\psi(\pi^{-1}W):\pi^{-1}W\cap\psi(\pi^{-1}W)]\geq s(\psi)\,.$$
It is now clear from the above proof that $\mathcal M(\pa(\phi)/\bik(\phi),\tilde\psi)=\{\pi U:U\in \mathcal M(\pa(\phi),\psi)\}$.

\smallskip\noindent
(5) follows from parts (3) and (4) using that $\nub(\phi)$ is contained in $\pa(\phi)$.
\end{proof}

As a consequence of the above lemma, we can define $\nub(\phi)$ without using the scale or minimizing subgroups. In fact, when $\phi$ is a topological automorphism,  Willis in \cite{Willis_nub} characterized $\nub(\phi)$ as the largest compact $\phi$-stable subgroup on which $\phi$ acts ergodically; equivalently, it is the largest compact $\phi$-stable subgroup with no proper relatively open $\phi$-stable subgroups. Using this, we obtain the following

\begin{corollary}
Let $G$ be a \tdlc group and $\phi\in \End(\phi)$. Then $\nub(\phi)$ is the largest compact $\phi$-stable subgroup of $G$ which contains $\bik(\phi)$ and such that, if $A\leq \nub(\phi)$ is a relatively open $\phi$-stable subgroup containing $\bik(\phi)$, then $A=\nub(\phi)$.
\end{corollary}
\begin{proof}
We have already noticed that $\nub(\phi)$ is a compact $\phi$-stable subgroup of $G$ which contains $\bik(\phi)$. Furthermore, given a relatively open $\phi$-stable subgroup $\bik(\phi)\leq A\leq \nub(\phi)$,  then $\pi A$ is a relatively open $\tilde\psi$-stable subgroup of $\pi(\nub(\phi))=\pi(\nub(\psi))=\nub(\tilde\psi)$ (see Lemma \ref{red_scale}), where $\pi:\pa(\phi)\to \pa(\phi)/\bik(\phi)$ is the natural projection. By \cite[Corollary 4.7]{Willis_nub}, $\pi A=\nub(\tilde\psi)$, so that $A=\nub(\phi)$ as desired. 

It remains to show that $\nub(\phi)$ is the largest subgroup with these properties. Indeed, given any compact $\phi$-stable subgroup $K$ of $G$, then $K\leq \pa(\phi)$. In fact, for every $x\in K$ there exists $(x_n)_{n\in\N}\subseteq K$, such that $x_0=x$ and $\phi(x_{n+1})=x_n$ for all $n\in\N$ (use that $\phi K=K$); moreover, the closure of $(x_n)_{n\in\N}$ is compact, being $K$ compact. Suppose also that $\bik(\phi)\leq K$ and that, given a relatively open $\phi$-stable subgroup $\bik(\phi)\leq A\leq K$, then $A=K$. This means that $\pi K$ is a compact $\tilde\psi$-stable subgroup of $\pa(\phi)/\bik(\phi)$ with no proper relatively open $\tilde\psi$-stable subgroups. As $\nub(\tilde\psi)$ is the largest subgroup of $\pa(\phi)/\bik(\phi)$ with this property, $\pi K\leq \nub(\tilde\psi)=\pi(\nub(\phi))$, and so $K\leq \nub(\phi)$.
\end{proof}

We conclude this subsection by giving a counterpart of Lemma \ref{red_scale} for the topological entropy:

\begin{lemma}\label{redtopar}
Let $G$ be a \tdlc group and $\phi\in\End(G)$. Then:
\begin{enumerate}[\rm (1)]
\item $h_{top}(\phi)=h_{top}(\psi)$;
\item $h_{top}(\bar\phi)=h_{top}(\bar \psi)$, where $\bar \phi$ and $\bar \psi$ are the maps induced by $\phi$ and $\psi$ respectively on $G/\nub(\phi)$ and $\pa(\phi)/\nub(\phi)$.
\end{enumerate}
\end{lemma}
\begin{proof}
We verify just (2), as the proof of (1) follows the same arguments. 
%
%
Let us start noticing that, by Proposition \ref{phiN},
\begin{align*}
h_{top}(\bar\psi)&=\sup\{\log[\phi M:M]:  \nub(\phi)\leq M\leq \pa(\phi),\ M\ \text{compact},\ M\leq \phi M,\ [\phi(M):M]<\infty\}\\
&\leq\sup\{\log[\phi M:M]:  \nub(\phi)\leq M\leq G,\ M\ \text{compact},\ M\leq \phi M,\ [\phi M:M]<\infty\}=h_{top}(\bar \phi)\,.
\end{align*}
On the other hand, by Corollary \ref{coro_lim_free} and Proposition \ref{phiN},
\begin{align*}
h_{top}(\bar\phi)&=\{\log[\phi U_+:U_+]: \nub(\phi)\leq U\in \B(G)\}\\
&\leq \sup\{\log[\phi M:M]:  \nub(\phi)\leq M\leq \pa(\phi),\ M\ \text{compact},\ M\leq \phi M,\ [\phi M:M]<\infty\}\\
&=h_{top}(\bar\psi)\,.\qedhere\end{align*}
\end{proof}

\subsection{The topological entropy knows all the values of the scale}

We give first the precise relation, stated in the Introduction in \eqref{s=h}, between the topological entropy and the scale:

\begin{proposition}\label{scale}
Let $G$ be a \tdlc group and $\phi\in\End(G)$. Then
$$\log s(\phi)=h_{top}(\bar \phi)\,,$$
where $\bar \phi:G/\nub(\phi)\to G/\nub(\phi)$ is the map induced by $\phi$.
\end{proposition}
\begin{proof}
Since $\mathcal M(G,\phi)$ is a base for the neighborhoods of $\nub(\phi)$ in $G$ by Lemma \ref{Mbase}, in view of Lemma \ref{basesuff}(3) we have $h_{top}(\bar \phi)=\sup\{H_{top}(\phi,U):U\in \mathcal M(G,\phi)\}$.
Furthermore, given $U\in\mathcal M(G,\phi)$, Proposition \ref{limit_free} and Lemma \ref{(B)}(2) give $H_{top}(\phi,U)=\log[\phi U_+:U_+]=\log s(\phi)$.
Thus, $h_{top}(\bar \phi)=H_{top}(\phi,U)=\log s(\phi)$, for any $U\in\mathcal M(G,\phi)$.
\end{proof}

As a consequence of Lemma \ref{red_scale}, Proposition \ref{scale} and Lemma \ref{redtopar}, we obtain
 $$\log s(\tilde\psi)=\log s(\psi)=\log s(\phi)=h_{top}(\bar\phi)=h_{top}(\bar\psi)\,.$$
 
Another consequence of Proposition \ref{scale} and Proposition \ref{phiN} is the following formula for the computation of the scale.

\begin{corollary}
Let $G$ be a \tdlc group and $\phi\in\End(G)$. Then
$$\log s(\phi)=\sup\{ \log[\phi M:M] : \nub(\phi) \leq M\leq G, M\ \text{compact}, M\leq \phi M, [\phi M:M]<\infty\}$$ 
\end{corollary}
 
Since $\nub(\phi)$ is a compact $\psi$-stable subgroup of $\pa(\phi)$ which contains $\ker(\psi)$ by \eqref{kpainnub}, Theorem \ref{AT} applies to $\psi$ and $\nub(\phi)$, so we have the following

\begin{corollary}\label{par}
Let $G$ be a \tdlc group and $\phi\in\End(G)$. Then $$h_{top}(\phi)=h_{top}(\phi\restriction_{\nub(\phi)})+h_{top}(\bar \phi).$$
\end{corollary}
\begin{proof}
By Lemma \ref{redtopar}(1) and Theorem \ref{AT}, $h_{top}(\phi)=h_{top}(\psi)=h_{top}(\psi\restriction_{\nub(\phi)})+h_{top}(\bar \psi)$. Since $\phi\restriction_{\nub(\phi)}=\psi\restriction_{\nub(\phi)}$, and since $h_{top}(\bar\psi)=h_{top}(\bar\phi)$ by Lemma \ref{redtopar}(2), we get the thesis.
\end{proof}

As a consequence of Corollary \ref{par} and Proposition \ref{scale} we obtain the following formula:
\begin{equation}\label{AT_scala}h_{top}(\phi)=\log s(\phi)+h_{top}(\phi\restriction_{\nub(\phi)})\,.\end{equation}
Applying this formula, we obtain a characterization of when $h_{top}(\phi)=\log s(\phi)$:

\begin{corollary}\label{last}
Let $G$ be a \tdlc group and $\phi\in\End(G)$. The following are equivalent:
\begin{enumerate}[\rm (1)]
\item $h_{top}(\phi)=\log s(\phi)$;
\item $\nub(\phi)=\{1\}$;
\item $h_{top}(\phi\restriction_{\nub(\phi)})=0$.
\end{enumerate}
\end{corollary}
\begin{proof}
It is clear that (2) implies (3), while (3) implies (1) by \eqref{AT_scala}. 
It remains to verify that (1) implies (2). Indeed,  if $\nub(\phi)\neq \{1\}$, there exists $U\in\B(G)$ not containing $\nub(\phi)$. By \cite[Proposition 3]{Willis_endo}, there exists $n\in\N$ such that $V:=U_{-n}$ is tidy above but, since $V$ does not contain $\nub(\phi)$, it is not tidy below, that is, it is not minimizing by \eqref{m=t}. Thus,
\begin{equation*}\log s(\phi)<\log [\phi V:V\cap \phi V]=\log [\phi V_+:V_+]=H_{top}(\phi,V)\leq h_{top}(\phi)\,.\qedhere\end{equation*}
\end{proof}

Since $s(\phi)\in\N_{>0}$, we obtain that $h_{top}(\phi)$ is finite whenever $\nub(\phi)=\{1\}$. More generally, applying Theorem \ref{AT}, we get
\begin{equation*}
h_{top}(\phi)=\infty\ \ \  \Longleftrightarrow \ \ \ h_{top}(\phi\restriction_{\nub(\phi)})=\infty\,.
\end{equation*}

\bibliographystyle{plain}

\end{document}